\documentclass[a4paper,11pt]{amsart}
\usepackage{eucal}
\usepackage{cite}
\usepackage{amsmath,amsthm,amssymb}
\usepackage{amsfonts}
\usepackage{latexsym}
\usepackage{iftex}
\ifpdf\else
  \ifXeTeX\else
    \PassOptionsToPackage{dvipdfmx}{hyperref}
  \fi
\fi
\usepackage{hyperref}
\usepackage{eucal}
\usepackage{mathrsfs}
\usepackage{txfonts}
\theoremstyle{plain}
\newtheorem{theorem}{Theorem}[section]
\newtheorem*{thm*}{Theorem}

\newtheorem{proposition}[theorem]{Proposition}

\newtheorem{lemma}[theorem]{Lemma}
\newtheorem{corollary}[theorem]{Corollary}

\theoremstyle{definition}
\newtheorem{definition}[theorem]{Definition}
\newtheorem{example}[theorem]{Example}
\makeatletter
\renewenvironment{proof}[1][\proofname]{\par
  \normalfont
  \topsep6\p@\@plus6\p@ \trivlist
  \item[\hskip\labelsep{\bfseries #1}\@addpunct{\bfseries.}]\ignorespaces
}{%
  \endtrivlist
}
\renewcommand{\proofname}{proof}
\theoremstyle{remark}
\newtheorem{remark}[theorem]{Remark}

\numberwithin{equation}{section}

\title{Tilting equivalence of finite almost derived algebraic cobordism for perfectoid algebras} 
\date{\today} 
\author{Yuki Kato}
\thanks{The author was supported by Grants-in-Aid for Scientific Research No.~23K03080, Japan Society for the Promotion of Science.}
\address{National Institute of Technology, Kurume College, 
	      1-1-1, Komorino, Kurume, Fukuoka, 830-8555 JAPAN}
\email{kato\_051@kurume-nct.ac.jp}

\newcommand{\A}{\mathbb{A}}
\newcommand{\Z}{\mathbb{Z}}
\newcommand{\K}{\mathbb{K}}

\newcommand{\Spec}{\mathrm{Spec}}

\newcommand{\Hom}{\mathrm{Hom}}

\newcommand{\Ext}{\mathrm{Ext}}

\newcommand{\sSet}{\mathrm{Set}_{\Delta}}

\newcommand{\CAlg}{\mathrm{CAlg}}

\newcommand{\Fun}{\mathrm{Fun}}
\newcommand{\Exc}{\mathrm{Exc}}

\newcommand{\tm}{\tilde{\mathfrak{m}}}

\newcommand{\MSp}{\mathrm{MSp}}
\newcommand{\Sp}{\mathrm{Sp}}

\newcommand{\MGL}{\mathrm{MGL}}
\newcommand{\E}{\mathbb{E}_{\infty}}

\usepackage{enumerate}
\usepackage{array}
\usepackage[all]{xy} 
\usepackage{delarray}
\def\qed{{\hfill $\Box$}}
\begin{document}
\thispagestyle{empty}
\begin{abstract}
In this paper, we prove tilting equivalence for the finite almost derived algebraic
cobordism spectrum $\mathrm{dMGL}^{a,\rm fin}$ of perfectoid algebras. More
precisely, if $V$ is an integral perfectoid valuation ring and $A$ is an
integral perfectoid $V$-algebra, then the tilting functor induces a weak
equivalence
\[
  \mathrm{dMGL}^{a,\rm fin}(A) \simeq \mathrm{dMGL}^{a,\rm fin}(A^\flat).
\]
This invariant is a finite syntomic, derived, and
non-$\mathbb{A}^1$-local version of algebraic cobordism, designed to
retain infinitesimal deformation data over mixed characteristic bases.
To prove the result, we first establish the corresponding finite
non-unital statement and isolate a form of excisive approximation for
pointed $\infty$-categories, including non-presentable ones. In the
locally finitely presentable case, this agrees with the framework of
Heuts. We also define approximation functors along natural
transformations and apply them to the comparison between periodic
algebraic cobordism and homotopy $K$-theory, obtaining Bott periodicity
and Gabber rigidity.
\end{abstract}

\maketitle
\section{Introduction}
\label{sec:introduction}
Motivic homotopy theory, as developed by Morel and Voevodsky~\cite{MV,VoeH}, is centered on $\mathbb{A}^1$-homotopy invariance. This principle serves as a primary source of the theory's strength; it yields robust invariants in algebraic geometry and leads to the algebraic cobordism spectrum $\MGL$ as well as other oriented cohomology theories. However, $\mathbb{A}^1$-homotopy invariance necessitates the d\'evissage property, which indicates that any such cohomology theory is equivalent to that of the underlying reduced subscheme. For instance, Weibel's homotopy $K$-theory $\mathrm{KH}$~\cite{weibel2} famously exhibits the d\'evissage property and completely annihilates the $K$-theory of nilpotent algebras. This implies that any $\mathbb{A}^1$-invariant cohomology theory remains completely insensitive to nilpotent ideals, resulting in identical evaluations on a scheme and its underlying reduced subscheme. Therefore, the consideration of $\mathbb{A}^1$-local objects inevitably involves the destruction of the cohomology theory of nilpotent non-unital algebras, including all nilpotent thickenings and infinitesimal deformation data. 

While this insensitivity often simplifies smooth geometry over a field, it significantly obstructs the study of mixed characteristic and perfectoid geometry. In these settings, phenomena are governed by adic topologies and towers of infinitesimal thickenings, making infinitesimal information an essential part of the geometric structure. A central problem is constructing cobordism-type invariants that retain essential deformation data by relinquishing strict $\mathbb{A}^1$-invariance. Derived algebraic cobordism $\mathrm{dMGL}$, introduced by Lowrey--Sch\"urg~\cite{LS16} and developed in a bivariant form by Annala~\cite{Ann21}, addresses this issue by utilizing quasi-smooth derived geometry instead of imposing $\mathbb{A}^1$-localization. The resulting theory effectively preserves the infinitesimal deformation data that classical motivic cobordism fails to capture. However, the unlocalized derived theory is significantly larger and more challenging to control directly. This paper's guiding idea is to extract a finite syntomic, non-unital fragment from this theory and study it using Goodwillie calculus.

Goodwillie calculus, introduced in~\cite{Goodwillie1,Goodwillie2,Goodwillie3},
may be viewed as a Taylor expansion of functors. Its linear part is
captured by excisive functors, and higher approximations form a tower
\[
  \cdots \to P^n(F) \to P^{n-1}(F) \to \cdots \to P^1(F) \to P^0(F).
\]
We apply this idea not only to functors, but also to pointed
$\infty$-categories. The category of non-unital algebras is naturally
pointed, with the zero algebra serving as both the initial and terminal
object, while square-zero extensions emerge as linear data in this pointed
context. Thus, the excisive approximation provides a method to isolate
and control infinitesimal deformation data without enforcing
$\mathbb{A}^1$-invariance. Our formulation is tailored for the
non-presentable categories that arise here; in the locally finitely
presentable case, it aligns with the Goodwillie calculus of
$\infty$-categories developed by Heuts~\cite{Heuts2018}, following
Lurie's $\infty$-categorical version of Goodwillie calculus~\cite[Chapter
6]{HA}.

The first main application is perfectoid geometry with almost mathematics. We define finite
non-unital derived algebraic cobordism $\mathrm{dMGL}^{\rm nu,f}$ and its almost
version $\mathrm{dMGL}^{a,\rm fin}$. For a perfectoid valuation ring $V$ and an
integral perfectoid $V$-algebra $A$, we prove that tilting preserves the
finite non-unital theory:
\[
 \mathrm{dMGL}^{\rm nu,f}(A) \simeq \mathrm{dMGL}^{\rm nu,f}(A^\flat)
\]
(Theorem~\ref{mainthm}). Transitioning to almost algebras reveals the primary
invariant emphasized in this paper. Specifically, if $V$ is an integral
perfectoid valuation ring, tilting induces a weak equivalence
\[
  \mathrm{dMGL}^{a,\rm fin}(A) \simeq \mathrm{dMGL}^{a,\rm fin}(A^\flat)
\]
for any integral perfectoid $V$-algebra $A$ (Corollary~\ref{almost-main}).

The same calculus also gives a second family of applications. Given a
natural transformation $\alpha:F\to G$, we define the approximation functors
$P^n_G(F)$ as the homotopy fiber of
\[
  G \to \mathrm{cok}\,\alpha \to P^n(\mathrm{cok}\,\alpha).
\]
We call $P^n_G(F)$ the Goodwillie approximation of $F$ to $G$. Applying
this construction to the canonical morphism from algebraic cobordism to
homotopy algebraic $K$-theory produces approximations of cobordism that
inherit key features of $K$-theory. In particular, we obtain Bott
periodicity and rigidity statements for these approximations, including
the lift of $K$-theoretic equivalences to periodic algebraic cobordism
(Theorem~\ref{theoremB}) and an analogue of Gabber rigidity
(Theorem~\ref{Rigidity}).

We briefly indicate the organization of the paper. In
Section~\ref{sec:Goodwillie}, we recall Goodwillie calculus for functors
and its higher excisive approximations. In Section~\ref{sec:n-excisive},
we develop the form of excisive approximation for pointed
$\infty$-categories necessary for non-unital algebras and compare it to
Heuts's theory in the finitely presentable case. In
Section~\ref{sec:perfectoid}, we construct the finite syntomic
non-unital and almost versions of derived algebraic cobordism and prove
the tilting invariance theorem and its almost corollary. Finally, in Section~\ref{sec:functor},
we introduce approximation functors along natural transformations and
apply them to the morphism $\MGL\to \K$, obtaining Bott periodicity and
Gabber rigidity for the resulting Goodwillie approximations.

\subsection*{Acknowledgments}
I would like to thank Masaki Hanamura for the helpful discussion about Sections~\ref{sec:Goodwillie} and~\ref{sec:n-excisive}. Our discussion generalizes the universal property of $n$-excisive approximation for $\infty$-categories related to functors that admit a right adjoint.
\paragraph{Use of AI tools.}
The author used Google's Gemini-pro 3.1 as a conversational research aid for brainstorming and refining mathematical formulations, and OpenAI's Prism (GPT-5.2) for editorial and expository assistance, including structural refinement. These tools improved the manuscript's clarity and indicated where abbreviated arguments warranted fuller exposition. All results, proofs, and final formulations were reviewed, verified, and approved by the author, who bears sole responsibility for the content of this work.

\section{A short review of Goodwillie calculus}
\label{sec:Goodwillie} In this section, following Lurie's
textbook~\cite[Chapter 6]{HA}, we recall the definition of Goodwillie
calculus for functors between $\infty$-categories. We then define higher
excisive objects of $\infty$-categories.  

The category $\sSet$ of simplicial sets has a standard simplicial model
structure called the {\it Kan--Quillen model structure}. The simplicial
nerve of $\sSet$ is denoted by $\mathcal{S}$, which is called the
$\infty$-category of {\it spaces} or {\it $\infty$-groupoids}.

Given a functor $F: \mathcal{C} \to \mathcal{D}$ of $\infty$-categories under suitable
conditions, one can regard excisive functors as linear approximations. In this
setting, the natural morphism $R^n(F) \to \Omega_{\mathcal{D}}\Sigma_{\mathcal{D}}R^n(F)$
is a weak equivalence, where $\Omega_{\mathcal{D}}$ and $\Sigma_{\mathcal{D}}$ denote
the loop and suspension functors, respectively.

For any finite set $S$,
$\mathbf{P}(S)$ denotes the set of
subsets of $S$, called the power set of $S$, and $\mathbf{P}_{ \le m}(S)$ denotes the set of subsets of $S$ whose cardinality is at most $m$. The sets $\mathbf{P}_{ <
m}(S)$, $\mathbf{P}_{ \ge m}(S)$, and $\mathbf{P}_{ > m}(S)$ are
defined similarly. 

Let $\mathcal{C}$ be an $\infty$-category. A functor $X: N(\mathbf{P}(S)) \to \mathcal{C}$ is called an $S$-cube. In the case $S=[n]=\{0,\,1,\,\ldots,\,n\}$, $S$-cubes are called $(n+1)$-cubes for any $n \ge -1$. Here, $[-1]$ is the empty set. We let $\mathrm{Cube}^m(\mathcal{C})$ denote the
$\infty$-category of $m$-cubes of $\mathcal{C}$ for any $m \ge 0$.  

\begin{definition}
Let $\mathcal{C}$ be an $\infty$-category and fix an integer $n \ge 0$. 
\begin{enumerate}[(1)]
\item An $n$-cube $X_\bullet$ is {\it Cartesian} if it is obtained as a right Kan extension of the cube $\mathrm{Cube}^n_{\ge 1 }(\mathcal{C})$.
\item An $n$-cube $X_\bullet$ is {\it strongly Cartesian} if it is obtained as
      a right Kan extension of the cube $\mathrm{Cube}^n_{ \ge n-1}(\mathcal{C})$.
\item An $n$-cube $X_\bullet$ is {\it coCartesian} if it is obtained as a left
      Kan extension of the cube $\mathrm{Cube}^n_{\le n-1}(\mathcal{C})$.
\item An $n$-cube $X_\bullet$ is {\it strongly coCartesian} if it is obtained
      as a left Kan extension of the cube $\mathrm{Cube}^n_{\le 1}(\mathcal{C})$.
\end{enumerate}
\end{definition}          
\begin{example}
Cartesian or coCartesian $0$-cubes are just weak equivalences, and any
$0$-cube is strongly Cartesian and coCartesian. Cartesian $1$-cubes and
strongly Cartesian $1$-cubes are homotopy Cartesian squares.
\end{example}

The definition of higher excisive functors is as follows: 
\begin{definition}[\cite{HA} p.1015, Definition 6.1.1.3]
Let $\mathcal{C}$ be an $\infty$-category admitting all finite colimits
and a final object, and $\mathcal{D}$ an $\infty$-category admitting all
finite limits. A functor $F: \mathcal{C} \to \mathcal{D}$ is {\it
$n$-excisive} if it sends strongly coCartesian $n$-cubes to Cartesian
$n$-cubes. In the $n = 1$ case, one-excisive functors are
simply called excisive functors. For any $n \ge 0$, let
$\mathrm{Exc}^n(\mathcal{C},\, \mathcal{D})$ denote the full subcategory
of $\Fun(\mathcal{C},\, \mathcal{D})$ spanned by $n$-excisive
functors.
\end{definition}

\begin{definition}[\cite{HA} p.1016, Definition 6.1.1.6]
An $\infty$-category $\mathcal{D}$ is {\it Goodwillie-differentiable} if
it admits finite limits and sequential colimits, and the sequential 
colimit functor
\[
   \varinjlim\colon \Fun(N(\Z_{\ge 0}),\,\mathcal{D}) \to \mathcal{D}
\]
commutes with finite limits. 
\end{definition}

\begin{example}
Any $\infty$-topos is Goodwillie-differentiable by definition. In particular,
the $\infty$-category $\mathcal{S}$ of spaces is Goodwillie-differentiable.
\end{example}

\begin{theorem}[\cite{HA} p.1016, Theorem 6.1.1.10]
\label{HATheorem6.1.1.10}
Let $\mathcal{C}$ be an $\infty$-category with a final object and
$\mathcal{D}$ a Goodwillie-differentiable $\infty$-category. Then, for
each $n \ge 0$, the inclusion functor
\[
\Exc^n(\mathcal{C},\, \mathcal{D}) \to \Fun(\mathcal{C},\,\mathcal{D})
\]
admits a left adjoint $P^n:\Fun(\mathcal{C},\,\mathcal{D}) \to
\Exc^n(\mathcal{C},\, \mathcal{D}) $, which is left exact. \qed
\end{theorem}

We recall the explicit construction of the $n$-excisive approximation $P^n(F)$ for a functor $F: \mathcal{C} \to \mathcal{D}$. 
For any finite subset $S$ of $[n]$, the $S$-pointed cone functor $C_S$ is the composition 
\[
 \mathcal{C} \simeq  \mathrm{Cube}^S_{\le 0}(\mathcal{C})  \overset{i_{0!}}{\to} \mathrm{Cube}^S_{\le 1}(\mathcal{C}) \overset{i_{\le 1 *}}{\to} \mathrm{Cube}^S(\mathcal{C}) 
 \overset{\mathrm{Ev}(S)}{\to} \mathcal{C}, 
\]
where $i_{0!}$ denotes the right Kan extension along the inclusion $i_0: \mathbf{P}_{\le 0}(S)  \to \mathbf{P}(S)$, $i_{\le 1 *}$ denotes the left Kan extension of $ \mathbf{P}_{\le 0}(S)  \to \mathbf{P}(S)$, and $\mathrm{Ev}(S)$ is the evaluation at the terminal vertex of cubes. Note that $C_\emptyset$ is the identity and $C_S$ is the final object of $\mathcal{C}$ whenever $|S|=1$.  
The functor $T_n(F): \mathcal{C} \to \mathcal{D}$ is defined as follows: 
\[
 T_n(F)(X) = \varprojlim_{\emptyset \neq S \subset [n]} F(C_S(X)) 
\]
for any $X \in \mathcal{C}$. The homotopy colimit of this sequence defines the $n$-excisive approximation $P^n$:
\[
    T_n(F) \to T_n(T_n(F)) \to T^3_n(F) \to \cdots \to T^k_n(F) \to \cdots 
\]
By the definition of $T_n$, if $F$ is $n$-excisive, the canonical map $\theta_F : F \to T_n(F)$ is a weak equivalence. Therefore, the colimit map $F \to P^n(F)$ is also a weak equivalence. The universal property in Theorem~\ref{HATheorem6.1.1.10} follows from the following lemmas.

\begin{lemma}[Rezk \cite{HA} p.1021, Lemma 6.1.1.26]
\label{Exc-lemma} Let $X: \mathbf{P}(S) \to \mathcal{C}$ be an $S$-cube and $F:\mathcal{C} \to \mathcal{D}$ be a functor. For any $I \subset S$, we define a cube $X_I : \mathbf{P}(S) \to \mathcal{C}$ by
\[
 X_I:    S' \mapsto   X(I \cup S') 
\]
Then, for any $\emptyset \neq  S' \in \mathbf{P}(S)$, the cube
\[
 F(X_\blacksquare(S')) :  I  \mapsto F(X_I(S'))     
\]
and the limit \[
Y_\blacksquare = \varprojlim_{S' \neq \emptyset}  F(X_\blacksquare(S'))  
\]
of cubes is Cartesian. Furthermore, if $X$ is strongly coCartesian, the map $\theta_F(X) : F(X) \to T_n(F)(X)$ factors through $Y_\blacksquare$. 
\end{lemma}
\begin{proof}
For any set $S' \subset S$, let $\mathbf{P}_{S'}(S)$ denote the set of subsets of $S$ that contain $S'$.  Then the projection $\pi_{S'} : \mathbf{P}(S) \to \mathbf{P}_{S'}(S)$ is defined by $\pi_{S'}(I)= I \cup S'$. The cube $F(X_\blacksquare(S'))$ is the inverse image of $\pi_{S'}$ of the restriction $F(X)|_{\mathbf{P}_{S'}(S)}: \mathbf{P}_{S'}(S) \to \mathcal{D}$. Therefore, the homotopy limit $\varprojlim_{I \neq \emptyset }F(X_I(S'))$ is weakly equivalent to the homotopy limit of $F(X)|_{\mathbf{P}_{S'}(S)}: \mathbf{P}_{S'}(S) \to \mathcal{D}$, which is $F(X_{S'}(S'))=F(X(S'))$. From $F(X_\emptyset(S'))=F(X(S'))$, we obtain that $F(X_\blacksquare(S'))$ is Cartesian, whenever $S'$ is nonempty, and the limit $Y_\blacksquare$ of Cartesian cubes is also Cartesian. 

The identity $F(X)|_{\mathbf{P}_{S'}(S)} \to F(X)|_{\mathbf{P}_{S'}(S)}  $ induces the unit morphism $  F(X) \to F(X_\blacksquare(S'))$, and $F(X) \to F(X_\blacksquare(\emptyset))$ is the identity. Therefore, there is a canonical map $F(X) \to Y_\blacksquare$ of cubes. We assume that $X$ is strongly coCartesian. Note that $X_I|_{\mathbf{P}_{\le 1}(S)}$ is the left Kan extension of the restriction $X|_{\mathbf{P}_{\le 1}(S \setminus I) }$ along 
\[
     i_I:  \mathbf{P}_{\le 1}(S \setminus I)  \to     \mathbf{P}_{\le 1}(S  ),   \quad i_I (S') = S' \cup I. 
\]
Hence, $X_I$ is strongly coCartesian, and for any $I \subset S$, the identity morphism $X_I(\emptyset) \to X(I)$ induces a map $X_I \to C_\blacksquare(X(I))$ of cubes, which is functorial for $I$. The induced map of limits $ \varprojlim_{S' \neq \emptyset} F(X_I(S') ) \to \varprojlim_{S' \neq \emptyset}F(C_{S'}(X(I)))$ is functorial for $I$, and we obtain the homotopically commutative square 
\[
\xymatrix@1{  F(X_\blacksquare(\emptyset)) \ar[d]  \ar[r]^\simeq & F(C_\emptyset(X)) \ar[d] \\
Y_\blacksquare \ar[r] & T_n(F)(X)
}
\]
of $S$-cubes, which implies

the desired property.  \qed
\end{proof}

\begin{lemma}[\cite{HA} p.1022, Lemma 6.1.1.33]
\label{n-app-lemma}
Let $\mathcal{C}$ be an $\infty$-category admitting finite colimits and a final object, $\mathcal{D}$ a Goodwillie-differentiable $\infty$-category, and $F: \mathcal{C} \to \mathcal{D}$ a functor. Then the $n$-excisive approximation $P^n(F): \mathcal{C} \to \mathcal{D}$ is $n$-excisive.
\end{lemma}
\begin{proof}
Let $X$ be a strongly coCartesian $n$-cube. By Lemma~\ref{Exc-lemma}, the
canonical map
\[
\theta_{T_n^{k-1}(F)}(X): T_n^{k-1}(F)(X) \to T_n^k(F)(X)
\]
factors through a Cartesian $n$-cube $Y_\blacksquare^k$. Therefore,
\[
P^n(F)(X) \simeq \varinjlim_k Y_\blacksquare^k.
\]
Since finite limits commute with filtered colimits,
$\varinjlim_k Y_\blacksquare^k$ is Cartesian. Hence $P^n(F)(X)$ is
Cartesian. \qed
\end{proof}
\begin{lemma}[\cite{HA}, p.1022, Lemma 6.1.1.34]
\label{LP}
Let $\mathcal{C}$ be an $\infty$-category which admits finite colimits and has a final object, and $\mathcal{D}$ a Goodwillie-differentiable $\infty$-category. 
Given a functor $F : \mathcal{C} \to \mathcal{D}$, the canonical map
$\theta:F \to T_n(F)$ induces an equivalence $P^n(F) \to P^n(T_n(F))$.
\qed
\end{lemma}
\begin{corollary}[\cite{HA}, p.1022, Lemma 6.1.1.35]
\label{TP}
The canonical map $P^n(F) \to P^n(P^n(F))$ is a weak equivalence.    \qed  
\end{corollary}

Let $\Fun(\mathcal{C},\,\mathcal{D})[T_n^{-1}]$ denote the localization of $\Fun(\mathcal{C},\,\mathcal{D})$ by the family of morphisms $\theta_F :F \to T_n(F)$. Lemma~\ref{LP} implies that the functor $P^n : \Fun(\mathcal{C},\,\mathcal{D}) \to \Exc^n(\mathcal{C},\,\mathcal{D})$ factors through $\Fun(\mathcal{C},\,\mathcal{D})[T_n^{-1}]$. Any object $F$ of $\Exc^n(\mathcal{C},\,\mathcal{D})$ satisfies $F \simeq T_n(F) \simeq P^n(F)$. Therefore, the $n$-excisive approximation induces a categorical equivalence
\[
P^n : \Fun(\mathcal{C},\, \mathcal{D})[T_n^{-1}] \to \Exc^n(\mathcal{C},\,\mathcal{D})
\]
of $\infty$-categories. Thus, we can identify the $n$-excisive approximation $P^n$ with the localization functor $\Fun(\mathcal{C}, \, \mathcal{D}) \to \Fun(\mathcal{C},\,\mathcal{D})[T_n^{-1}]$.

\section{Higher excisive approximations of \texorpdfstring{$\infty$-categories}{infinity-categories}}
\label{sec:n-excisive} In this section, using the $\infty$-categorical
Yoneda lemma, we define $n$-excisive objects of
$\infty$-categories. First, in Section~\ref{sec:Yoneda}, we recall the
$\infty$-categorical Yoneda lemma. Next, in Section~\ref{sec:app}, we
define $n$-excisive objects and prove their universal property. Finally,
in Section~\ref{sec:excisive}, we note that locally presentable excisive
$\infty$-categories are stable.
\subsection{The \texorpdfstring{$\infty$-categorical Yoneda lemma and $n$-excisive $\infty$-categories}{infinity-categorical Yoneda lemma and n-excisive infinity-categories}}
\label{sec:Yoneda}
\begin{lemma}[The \texorpdfstring{$\infty$}{infinity}-categorical Yoneda lemma~\cite{HT} p.317,
Proposition 5.1.3.1] \label{presheaves} Let $\mathcal{C}$ be a small
$\infty$-category. Then the functor
\begin{align*}
 \mathbb{Y}_\mathcal{C}: \mathcal{C} &\to \Fun(\mathcal{C}^{\rm op},\, \mathcal{S})\\
  X    &\mapsto \Hom_\mathcal{C}(-,\,X)    
\end{align*}
is fully faithful. \qed 
\end{lemma}

We write $\mathrm{Pre}(\mathcal{C}) =\Fun(\mathcal{C}^{\rm op},\,
\mathcal{S})$ for the $\infty$-category of {\it presheaves} on
$\mathcal{C}$, and the fully faithful functor $\mathbb{Y}$ is called the
{\it Yoneda functor (or Yoneda embedding)}. Dually, the functor
\begin{align*}
 \mathbb{Y}^\vee_\mathcal{C}: \mathcal{C} &\to \Fun(\mathcal{C},\, \mathcal{S})\\
  X    &\mapsto \Hom_\mathcal{C}(X,\,-)    
\end{align*}
is also fully faithful, which is called the coYoneda functor. 

\begin{definition}
\label{cpt}
An $\infty$-category $\mathcal{C}$ is {\it compactly generated} if $\mathcal{C}$ satisfies the following:
\begin{itemize}
    \item[(1)] There exists a small set $S$ of objects such that, for each $A \in S$, the functor $\Hom_\mathcal{C}(A,\, - ): \mathcal{C} \to \mathcal{S}$ preserves all filtered colimits.
    \item[(2)] A morphism $f:X \to Y$ is a weak equivalence if and only if $\Hom_\mathcal{C}(A,\,f): \Hom_\mathcal{C}(A,\, X) \to \Hom_\mathcal{C}(A,\,Y)$ is a weak equivalence for each $A \in S$. 
\end{itemize}
\end{definition}
We say that such $S$ in Definition~\ref{cpt} is a generator of
$\mathcal{C}$. For any $\infty$-category $\mathcal{C}$, $\mathcal{C}_0$
denotes the full subcategory spanned by compact objects.  By the
definition, if $\mathcal{C}$ is compactly generated, the $\mathcal{C}_0
\to \mathcal{C}$ induces a categorical equivalence
$\mathrm{Ind}(\mathcal{C}_0) \to \mathcal{C}$, where
$\mathrm{Ind}(\mathcal{C}_0)$ denotes the ind-category whose objects are
filtered colimits of objects of $\mathcal{C}_0$. Along the Yoneda embedding
$\mathbb{Y}_{\mathcal{C}_0} : \mathcal{C}_0 \to
\mathrm{Pre}(\mathcal{C}_0)$, the ind-category
$\mathrm{Ind}(\mathcal{C}_0)$ is identified with the full subcategory of
$\mathrm{Pre}(\mathcal{C})$ spanned by left exact presheaves.

An $\infty$-category that is both locally presentable and compactly generated is called {\it locally finitely presentable}.  

\subsection{The definition of higher excisive objects}
\label{sec:app} Let $F: \mathcal{C} \to \mathcal{D}$ be a functor
between pointed $\infty$-categories. The induced functor
$F_*:\Fun(\mathcal{C}^{\rm op},\, \mathcal{S} ) \to
\Fun(\mathcal{D}^{\rm op},\, \mathcal{S} )$ is the left Kan
extension of $\mathbb{Y}_\mathcal{D} \circ F: \mathcal{C} \to
\Fun(\mathcal{D}^{\rm op},\, \mathcal{S} )$ along
$\mathbb{Y}_\mathcal{C}:\mathcal{C} \to \Fun(\mathcal{C}^{\rm op},\,
\mathcal{S} )$. This correspondence gives an equivalence of
$\infty$-categories
\[
  \Fun^{\rm L}\left( \mathrm{Pre}(\mathcal{C}),\,  
\mathrm{Pre}(\mathcal{D}) \right) 
\to \Fun \left( \mathcal{C},\, \mathrm{Pre}(\mathcal{D})  \right), 
\]
where $\Fun^{\rm L}\left( \mathrm{Pre}(\mathcal{C}),\,
\mathrm{Pre}(\mathcal{D}) \right)$ denotes the $\infty$-category of
functors admitting a right adjoint. (See \cite[p.324, Theorem
5.1.5.6]{HT}). Furthermore, if $F: \mathcal{C} \to \mathcal{D}$ is a
functor between compactly generated $\infty$-categories sending compact
objects to compact objects, $F$ is the left Kan extension of its
restriction to the full subcategory $\mathcal{C}_0$ spanned by compact
objects, and admits a right adjoint $F^*: \mathcal{D} \to \mathcal{C}$,
which is the restriction of the inverse image functor
$F^{-1}:\mathrm{Pre}(\mathcal{D}_0) \to \mathrm{Pre}(\mathcal{C}_0)$.

\begin{definition}
Let $\mathcal{C}$ be a small $\infty$-category with finite colimits and
a final object. For any $n \ge 0$, we say that $X$ is an {$n$-excisive
object} if the coYoneda functor $\mathbb{Y}^\vee(X)$ is $n$-excisive.
Let $P^n(\mathcal{C})$ denote the Cartesian product
\[
   \mathrm{Exc}^n ( \mathcal{C},\, \mathcal{S}) 
\times_{ \Fun ( \mathcal{C},\, \mathcal{S})  }  \mathcal{C}
\]
along the coYoneda functor $\mathbb{Y}^\vee: \mathcal{C} \to \Fun(
\mathcal{C},\, \mathcal{S} )$. That is, the $\infty$-category $P^n
(\mathcal{C})$ is the full subcategory of $\mathcal{C}$ spanned by
$n$-excisive objects.
\end{definition}
Since any corepresentable functor is left exact, the coYoneda embedding $\mathbb{Y}^\vee_\mathcal{C} : \mathcal{C} \to \Fun(\mathcal{C},\,\mathcal{S})$ factors through the full subcategory $\Fun^{\mathrm{lex}}(\mathcal{C},\,\mathcal{S})$ spanned by left exact functors. Hence, $P^n (\mathcal{C})$ coincides with 
\[ 
\mathrm{Exc}^{n \mathrm{lex}} ( \mathcal{C},\, \mathcal{S}) 
\times_{ \Fun^\mathrm{lex} ( \mathcal{C},\, \mathcal{S})  }  \mathcal{C},
\]
where $\mathrm{Exc}^{n \mathrm{lex}} ( \mathcal{C},\, \mathcal{S})$ is the full subcategory spanned by left exact $n$-excisive functors.

The Yoneda functor $\mathbb{Y}_\mathcal{C}$ preserves all small limits. By the definition of $T_n$, one has a weak equivalence:
\[
T_n(\mathbb{Y}_\mathcal{C})(X) = \varprojlim_{\emptyset \neq S} \mathbb{Y}_\mathcal{C}(C_S(X)) \simeq \mathbb{Y}_\mathcal{C}( \varprojlim_{\emptyset \neq S}C_S (X)     )= \mathbb{Y}_\mathcal{C}(T_n (\mathrm{id}_\mathcal{C})(X)),   
\]
implying that the natural map $T_n \circ \mathbb{Y}_\mathcal{C} \to \mathbb{Y}_\mathcal{C} \circ T_n$ is a weak equivalence.
Furthermore, if $\mathcal{C}$ is compactly generated, the homotopy colimit map $P^n \circ \mathbb{Y}_\mathcal{C} \to \mathbb{Y}_\mathcal{C} \circ P^n$ is also a weak equivalence.

\begin{proposition}
\label{n-Excisive} Let $\mathcal{C}$ be a pointed small
$\infty$-category. For any $n \ge 0$, the following conditions are
equivalent:
\begin{itemize}
 \item[(1)] The fully faithful functor $P^n(\mathcal{C}) \to \mathcal{C}$ is
       a categorical equivalence.
 \item[(2)] The identity functor $\mathrm{id}_\mathcal{C}$ is $n$-excisive.
 \item[(3)]    The Yoneda functor $\mathbb{Y}_\mathcal{C}: \mathcal{C} \to \Fun(
       \mathcal{C}^{\rm op},\, \mathcal{S})$ is $n$-excisive.
 \item[(4)] Further, if $\mathcal{C}$ is compactly generated, the identity on $\mathcal{C}_0$ is $n$-excisive, where $\mathcal{C}_0$  denotes the full subcategory spanned by compact objects. 
     \end{itemize}
\end{proposition}
\begin{proof}
Let $Y$ be a strongly coCartesian $n$-cube. Since $\Hom_\mathcal{C}(X,\,-)$
commutes with all small limits, the $\infty$-categorical Yoneda lemma
implies that conditions (1) and (2) are equivalent to the statement that
$Y(\emptyset) \to \varprojlim_{\emptyset \neq S} Y(S)$ is a weak equivalence.
Hence $Y$ is also Cartesian. Since the Yoneda functor
$\mathbb{Y}_\mathcal{C}$ preserves all small limits, (2) and (3) are
equivalent.

We assume that $\mathcal{C}$ is compactly generated. To prove implication (4)$\Rightarrow$(2), it suffices to verify that $\mathrm{id}_\mathcal{C} \to P^n(\mathrm{id}_\mathcal{C}): \mathcal{C} \to \mathcal{C}$ is a weak equivalence. Any object $X$ of $\mathcal{C}$ is homotopically equivalent to a filtered colimit of compact objects. Since filtered colimits commute with finite limits and small colimits, Condition (4) implies Condition (2). \qed   
\end{proof}

For any object $X$ of the essential image of $P^n(\mathrm{id}_\mathcal{C}): \mathcal{C} \to \mathcal{C}$, the induced map $X \to T_n(X)$ is also a weak equivalence. 
Therefore, the essential image of $P^n(\mathrm{id}_\mathcal{C})$ can be identified with the localization $\mathcal{C}[T_n^{-1}]$.        

\begin{corollary}
\label{P-id}
Let $\mathcal{C}$ be a Goodwillie differentiable $\infty$-category that admits filtered colimits and a final object. The $n$-excisive approximation $P^n(\mathrm{id}_\mathcal{C})$ of the identity functor factors through $P^n(\mathcal{C})$, and the essential image of $P^n(\mathrm{id}_\mathcal{C})$ is categorically equivalent to $P^n(\mathcal{C})$. 
\end{corollary}
\begin{proof}
By Proposition~\ref{n-Excisive}, the embedding $P^n(\mathcal{C}[T_n^{-1}]) \to \mathcal{C}[T_n^{-1}]$ is a categorical equivalence. Hence $\mathcal{C}[T_n^{-1}]$ is contained in the full subcategory $P^n(\mathcal{C})$. \qed   
\end{proof}

\begin{theorem}
Let $\mathcal{C}$ be a Goodwillie differentiable $\infty$-category that admits filtered colimits and a final object, and $\mathcal{C}^\omega$ denote the full subcategory spanned by compact objects. 
\end{theorem}

\begin{proposition}
\label{adjoint} Let $\mathcal{C}$ and $\mathcal{D}$ be pointed
$\infty$-categories admitting finite limits and colimits, and $F:
\mathcal{C} \to \mathcal{D}$ a pointed functor admitting a right adjoint $F^*$. 
If $\mathcal{D}$ is
$n$-excisive, both $F:\mathcal{C} \to \mathcal{D}$ and $
F^*: \mathcal{D} \to \mathcal{C}$ are $n$-excisive.
\end{proposition}
\begin{proof}
Let $X_\bullet$ be a strongly coCartesian $n$-cube of $\mathcal{C}$. Since $F$ preserves all small colimits, $F(X_\bullet)$ is also strongly coCartesian. Therefore, by Lemma~\ref{n-Excisive}, $F(X_\bullet)$ is Cartesian. 

Let $Y_\bullet$ be a strongly coCartesian $n$-cube. Then $Y_\bullet$ is also Cartesian. Since the right adjoint preserves all small limits, $F^*(Y_\bullet)$ is Cartesian. 
\qed 
\end{proof}
A restriction of an $n$-excisive functor is also $n$-excisive.  
\begin{corollary}
Let $\mathcal{C}$ and $\mathcal{D}$ be locally finitely presentable $\infty$-categories, and let $F: \mathcal{C} \to \mathcal{D}$ be a pointed functor preserving compact objects. If $\mathcal{D}$ is $n$-excisive, $F$ and its right adjoint are also $n$-excisive.  \qed      
\end{corollary}

For any pointed $\infty$-category $\mathcal{C}$ admitting finite
colimits, the $n$-excisive approximation $P^n(\mathcal{C})$ has the following universal property:
\begin{theorem}
\label{Universality} Let $\mathcal{C}$ and $\mathcal{D}$ be pointed
$\infty$-categories admitting finite limits and colimits. If $\mathcal{D}$ is
$n$-excisive, then any pointed functor $F:\mathcal{C} \to \mathcal{D}$
admitting a right adjoint $F^*: \mathcal{D} \to \mathcal{C}$ factors through the $n$-excisive localization $\mathcal{C}[T_n^{-1}]$
as
\[
F: \mathcal{C} \overset{P^n(\mathrm{id}_\mathcal{C})}{\to} \mathcal{C}[T_n^{-1}] \to \mathcal{D}.
\]
\end{theorem}
\begin{proof}
It is sufficient to prove that the natural transformation $\mathrm{id}
\to P^n$ induces a weak equivalence
\[
  (F \circ \mathrm{id}_{\mathcal{C}} \to   F \circ P^n( \mathrm{id}_{\mathcal{C}} ) )  :   \mathcal{C} \to \mathcal{D}.
\] 
Since the restriction $F: \mathcal{C}
\to \mathcal{D} $ preserves all small colimits and final objects, one has
\[
  P^n(F^{*} \circ F  ) \simeq P^n(F^*) \circ F   \simeq  F^{*} \circ F  
\]  
by Proposition~\ref{adjoint}, which implies that the endofunctor $F^{*} \circ F:
\mathcal{C} \to \mathcal{C} $ is
$n$-excisive.  Therefore, the unit
\[
u:\mathrm{id}_{\mathcal{C}} \to F^{*} \circ F 
\] of the
induced adjunction $ F : \mathcal{C} \rightleftarrows
\mathcal{D} : F^{*}$ factors through the $n$-excisive
approximation $P^n(\mathrm{id}_{\mathcal{C}})$, yielding the induced natural transformation
$F(\mathrm{id}_{\mathcal{C}}) \to
F(P^n(\mathrm{id}_{\mathcal{C}}))$
which is homotopically split. Since the $n$-excisive approximation $P^n$ is homotopically
idempotent, the natural transformation
$F(P^n(\mathrm{id}_{\mathcal{C}})) \to
F(P^n(P^n(\mathrm{id}_{\mathcal{C}})))$ is a weak
equivalence, implying that its retract
$F(\mathrm{id}_{\mathcal{C}}) \to
F(P^n(\mathrm{id}_{\mathcal{C}}))$ is also a weak equivalence.  \qed
\end{proof}
By Lurie~\cite[p.324, Theorem 5.1.5.6]{HT}, we have the following:     
\begin{corollary}
\label{Universal-cpt}
Let $\mathcal{C}$ and $\mathcal{D}$ be pointed locally finitely presentable $\infty$-categories, and let $\mathcal{C}_0$ denote the full subcategory spanned by compact objects. Assume that $\mathcal{D}$ is $n$-excisive. Then any pointed functor $F_0:\mathcal{C}_0 \to \mathcal{D}$ admits a left Kan extension $F: \mathcal{C} \to \mathcal{D}$, which factors through the $n$-excisive localization $\mathcal{C}[T_n^{-1}]$ as in Theorem~\ref{Universality}. \qed
\end{corollary}

\begin{remark}
Let $\mathcal{C}$ be a locally finitely presentable
$\infty$-category and $\mathcal{D}$ an $n$-excisive
$\infty$-category. Let $\Fun^{\rm L}_\omega(\mathcal{C},\,\mathcal{D})$ denote the $\infty$-category of functors admitting a right adjoint and preserving compact objects. Then $P^n: \Fun^{\rm L}_\omega( \mathcal{C},\,
\mathcal{C} ) \to \Exc^{\mathrm{L},\,n}_\omega(
\mathcal{C},\, \mathcal{C} )$ is a weak
$n$-excisive approximation in the sense of Heuts~\cite[Definition
1.2]{Heuts2018}, where $\Exc^{\mathrm{L},\,n}_\omega(
\mathcal{C},\, \mathcal{C})$ denotes the
full subcategory of $\Fun^{\rm L}( \mathcal{C},\,
\mathcal{C} ) $ spanned by $n$-excisive functors. Note
that $\Exc^{\mathrm{L},\,n}( \mathcal{C},\,
\mathcal{C} )$ is presentable. For any colimit preserving
and pointed functor $f:\mathcal{C} \to
\mathcal{C}$, one has a chain of weak equivalences $P^n(f)
\to P^n( \mathrm{id}_{\mathcal{C}} \circ f) \to P^n(
\mathrm{id}_{\mathcal{C}}) \circ f$ by Lurie~\cite[p.~1021,
Remark 3.1.1.30]{HA}. In particular,
$P^n(\mathrm{id}_{\mathcal{C}}) $ is homotopically
idempotent, implying that $f$ is $n$-excisive if and only if $f \to
P^n(\mathrm{id}_{\mathcal{C}}) \circ f $ is
a weak equivalence. Equivalently, the essential image of
$P^n(\mathrm{id}_{\mathcal{C}})$ exhibits the $n$-excisive
approximation $P^n(\mathcal{C})$ in the equalizer:
\[
\xymatrix@1{
     \Exc^{\mathrm{L},\,n}(
\mathcal{C},\, \mathcal{C} )  \ar[r] &    \Fun^{\mathrm{L}}(
\mathcal{C},\, \mathcal{C} )   \ar[rr]<0.5mm>^{P^n(\mathrm{id}_{\mathcal{C}}) }  \ar[rr]<-0.5mm>_{\mathrm{id}_{\mathcal{C}}} & &  \Fun^{\mathrm{L}}(
\mathcal{C},\, \mathcal{C} ).
}
\]
\end{remark}

\begin{proposition}
Let $\mathcal{C}$ be a locally finitely presentable $\infty$-category admitting a final object, and $\mathcal{C}_0$ denote the full subcategory spanned by compact objects. Then $\mathcal{C}[T_n^{-1}] \to \Exc^{n\,\mathrm{lex}}(\mathcal{C}_0,\,\mathcal{S})$ is a categorical equivalence.      
\end{proposition}
\begin{proof}
Note that any object of $\mathcal{C}$ is identified with a filtered colimit of corepresentable functors on compact objects, each of which is left exact.
Let $X$ be an $n$-excisive object. Then one has a chain of weak equivalences
\[
 P^n(X) \simeq P^n(\varinjlim \Hom_{\mathcal{C}_0}(C_\alpha,\, - )) \simeq \varinjlim \Hom_{\mathcal{C}_0}(C_\alpha,\, P^n(-) ).
\]
Therefore, $X \to P^n(X)$ is a weak equivalence if and only if $X$ is a filtered colimit of objects of $P^n(\mathcal{C}_0)$. \qed
\end{proof}

\begin{theorem}
Let $\mathcal{C}$ be a locally finitely presentable $\infty$-category admitting a final object. Then the full subcategory $P^n(\mathcal{C})$ spanned by $n$-excisive objects has the following universal property: if $\mathcal{D}$ is a locally finitely presentable $\infty$-category that is $n$-excisive, then any pointed functor $F_0:\mathcal{C}_0 \to \mathcal{D}$ admits a left Kan extension $F: \mathcal{C} \to \mathcal{D}$ that factors through the $n$-excisive localization $P^n(\mathcal{C})$. \qed  
\end{theorem}

\subsection{Stable and excisive \texorpdfstring{$\infty$-categories}{infinity-categories}}
\label{sec:excisive}
Let $\mathcal{C}$ be an $\infty$-category admitting finite colimits and
a final object, $\mathcal{D}$ a Goodwillie-differentiable
$\infty$-category, and $F: \mathcal{C} \to \mathcal{D}$ an excisive
functor. For any object $X$ of $\mathcal{C}$, the canonical morphism
$F(X) \to \Omega_{\mathcal{D}} F (\Sigma_\mathcal{C} X)$ is a weak
equivalence in $\mathcal{D}$. Therefore,
\[ 
\Omega_{\mathcal{D}} \circ
(-): \mathrm{Exc}^1(\mathcal{C},\, \mathcal{D} ) \to
\mathrm{Exc}^1(\mathcal{C},\, \mathcal{D} )
\] 
  is a categorical equivalence, whose inverse is $(-)\circ
  \Sigma_{\mathcal{C}}$. In addition, the stabilization functor
  $\Sigma_{\mathcal{D}}^\infty: \mathcal{D} \to \mathrm{Sp}(\mathcal{D})
  $ induces a categorical equivalence
\[ 
\Sigma_{\mathcal{D}}^\infty
  \circ (-): \mathrm{Exc}^1(\mathcal{C},\, \mathcal{D} ) \to
  \mathrm{Exc}^1(\mathcal{C},\, \mathrm{Sp}(\mathcal{D}) ), 
\]
whose inverse is induced by the forgetful functor
$\Omega^\infty_{\mathcal{D}}: \mathrm{Sp}(\mathcal{D}) \to \mathcal{D}$.

Applying this to the case $\mathcal{D} = \mathcal{S}_{*}$, one has the
following:
\begin{proposition}
\label{Excisive} Let $\mathcal{C}$ be a pointed $\infty$-category
admitting finite limits and colimits. The following conditions are
equivalent:
\begin{enumerate}[(1)]
\item The $\infty$-category $\mathcal{C}$ is excisive.
\item For any integer $n \ge 0$, the unit transformation
      $\mathrm{Id}_\mathcal{C} \to \Omega^n_\mathcal{C} \circ
      \Sigma_{\mathcal{C}}^n$ is a weak equivalence.
\item The stabilization $\Sigma_+^\infty: \mathcal{C} \to
      \mathrm{Sp}(\mathcal{C})$ induces an equivalence $
      \mathrm{Id}_\mathcal{C} \to \Omega^\infty_+ \circ \Sigma_+^\infty$.
\end{enumerate}
\end{proposition}
\begin{proof}
First, we assume that $\mathcal{C}$ is excisive. Then, for any integer
$n \ge 0$, the unit 
\[ 
\mathrm{Id} \to \Omega_\mathcal{C}^n \circ \Sigma_\mathcal{C}^n :
\mathcal{C} \to \mathcal{C} 
\]
is a weak equivalence by Lemma~\ref{presheaves}.

In the case of condition (2), the stabilization functor
\[
\Sigma_+^\infty \mathcal{C} \to \mathrm{Sp}(\mathcal{C}) 
\] is fully faithful. Indeed,
for any $X$ and $Y$,
 \[
\Hom_{\Sp(\mathcal{C})}( \Sigma_+^\infty X ,\,
\Sigma^\infty_+ Y) \simeq \varinjlim_{n } \Hom_\mathcal{C} (\Sigma^n X
,\, \Sigma^n Y )
\] 
is weakly equivalent to the constant space $\Hom_\mathcal{C} ( X ,\,
Y)$. Again by Lemma~\ref{presheaves}, the unit $Y
\to \Omega_+^\infty( \Sigma_+^\infty(Y))$ is a weak equivalence.

Finally, we assume that the condition $(3)$ holds. Then, for any object
$X$ of $\mathcal{C}$, the coYoneda functor
\[ 
\mathbb{Y}^\vee(X) \simeq
(\Sigma^\infty_+)_*( \mathbb{Y}^\vee( \Sigma_+^\infty X) )=
\Hom_{\Sp(\mathcal{C})}( \Sigma_+^\infty X ,\, \Sigma^\infty_+ (-)):
\mathcal{C} \to \Sp(\mathcal{C}) \to \mathcal{S}
\] 
is excisive. 
 \qed
\end{proof}


\section{Application to the finite derived algebraic cobordism of perfectoid algebras}
\label{sec:perfectoid}

For a locally presentable category $\mathcal{C}$ with a final object, the functor category $\sSet^{\mathcal{C}^\omega}$ has the covariant model structure representing the $\infty$-category $\Fun^{\rm L}(N(\mathcal{C}),\,\mathcal{S})$. By \cite[p.906, Proposition 3.7.8]{HT}, the $n$-excisive approximation $P^n :\Fun^{\rm L}(\mathcal{C}, \, \mathcal{S}) \to \mathrm{Exc}^n ( N(\mathcal{C}),\, \mathcal{S})$ induces a Bousfield localization $P^{n}:\sSet^{\mathcal{C}^\omega} \to P^n(\sSet^{\mathcal{C}^\omega})$. Since $P^n$ has a right adjoint, it is accessible, and $P^n$ sends compact objects to compact objects. Hence, the $n$-excisive category $P^n(\mathcal{C}^\omega)$ is homotopically equivalent to the full subcategory of $P^n(\sSet^{\mathcal{C}^\omega})$ spanned by compact objects.

In this section, we consider $\E$-rings. Ordinary rings, whose homotopy groups are concentrated in degree zero, are called discrete $\E$-rings, or simply discrete rings. 

First, in Section~\ref{sec:finite syntomic}, to formulate a finite syntomic analogue of derived algebraic cobordism for non-unital $\E$-rings, we define finite syntomic (i.e., finite quasi-smooth) morphisms of $\E$-rings. In Section~\ref{sec:non-unital}, we review the excisive approximation of the pointed category of augmented commutative algebras. In Section~\ref{sec:MGL}, we verify that this finite derived algebraic cobordism preserves tilting equivalences between integral perfectoid algebras (Theorem~\ref{mainthm}).

\subsection{Finite syntomic morphisms of \texorpdfstring{$\E$-rings}{E-rings}}
\label{sec:finite syntomic}
We fix the base commutative unital ring $V$, which is a discrete $\E$-ring. 
The category $\CAlg_{V//V}$ of augmented commutative $V$-algebras is a pointed category, whose initial and terminal object is $V$. 
The operadic Smith-ideal theory~\cite[Theorem 4.4.1]{WY2024} induces an adjoint categorical equivalence
\[
  \mathrm{cok}:   \CAlg^{\rm nu}_{V}(\mathcal{C}) \rightleftarrows    \CAlg_{V//V}(\mathcal{C}): \ker
\]
of $\infty$-categories (see \cite[Corollary 5.11]{kato2023mathematics}). 
We define cotangent complexes by the language of Goodwillie calculus and Smith ideal theory. 
\begin{definition}[{\rm cf.}~\cite{HA} p.1296, Definition 7.3.2.14]
Let $\mathcal{C}$ be a symmetric monoidal $\infty$-category admitting finite colimits and a final object, $V$ a monoidal unit object of $\mathcal{C}$, and $\CAlg(\mathcal{C})_{V//V}$ the subcategory of $\mathcal{C}$ spanned by augmented commutative algebra objects. For an augmented $V$-algebra $P^1(\varepsilon): A \to V$, let $L_{A}$ denote the kernel of the excisive approximation $\varepsilon : P^1(A) \to V$. We call $L_A$ the {\it cotangent complex} of $A$.
\end{definition}

In derived algebraic geometry, intersection-theoretic data is unconditionally preserved by replacing flat transversality with quasi-smoothness. 
\begin{definition}
Let $A \to B$ be a morphism of augmented discrete $V$-algebras. We say that $B$ is {\it finite syntomic} over $A$ if it satisfies the following conditions:
\begin{enumerate}[(1)]
\item The $A$-algebra $B$ is a flat $A$-algebra and a finitely generated $A$-module.
\item The relative cotangent complex $L_{B/ A}$ is perfect and has Tor-amplitude in degrees $[-1,\,0]$.  
\end{enumerate}
\end{definition}
 
The definition of flatness of $\E$-algebras is described as follows: 
\begin{definition}[\cite{HA}, p.1240, Definition 7.2.2.10]
An $\E$-algebra morphism $A \to B$ is flat if 
\begin{itemize}
 \item[(1)] The homotopy group $\pi_0(B)$ is a flat $\pi_0(A)$-algebra.
 \item[(2)] For each $n \in \Z$, the induced morphism
\[  
    \pi_0(B) \otimes_{\pi_0(A)} \pi_n (A)  \to \pi_n(B)      
\]
is an isomorphism. 
\end{itemize}
\end{definition}
By definition, if $A$ is connected (resp. discrete), any flat $A$-algebra is connected (resp. discrete). 

We define finite syntomic morphisms for $\E$-rings, which correspond to finite quasi-smooth morphisms. Given an $\E$-ring $A$, let $\CAlg_{A}^{\flat}$ denote the full subcategory of $\CAlg_{A}$ spanned by flat $A$-algebras. 
\begin{definition}
A morphism $f: A \to B$ of $\E$-rings is said to be {\it finite syntomic} if $f$ belongs to $ \CAlg_{A}^{\rm FSyn}$. Namely, $f$ satisfies the following:
  \begin{itemize}
 \item[(1)] The $\E$-algebra morphism $f$ is flat.
 \item[(2)] The induced $V$-algebra homomorphism $\pi_0(f): \pi_0(A) \to \pi_0(B)$ is finite syntomic.  
 \end{itemize} 
\end{definition}
In other words, the $\infty$-category $\CAlg_{A}^{\rm FSyn}$ of finite syntomic $A$-algebras is defined by the homotopy Cartesian product
\[
  \CAlg_{A}^{\rm FSyn}= \CAlg_{A}^{\flat} \times_{\CAlg_{\pi_0(A)}^{\flat}  }   \CAlg_{\pi_0(A)}^{\rm FSyn}. 
\] 

With this definition, finite syntomic morphisms of $\E$-rings are locally represented as pullbacks of universal finite syntomic morphisms $p_d: V_d \to U_d$ $(d \ge 1)$ (see the Stacks Project~\cite{stacks-project}):
\begin{proposition}[Stacks Project, Lemma 49.11.7]
\label{univ-fsyn} Let $f:Y \to X$ be a finite syntomic morphism of (discrete) $\Spec{V}$-schemes. Then for any $x \in X$ there exists an integer $d \ge 1$ and a commutative diagram
\[
  \xymatrix@1{
  Y \ar[d]_f & \ar[l] X' \ar[r] \ar[d] & V_d \ar[d]_{p_d} \ar[dr] & \\ 
  X &  \ar[l] Y' \ar[r] & U_d \ar[r] & \Spec{V}
}
\] 
with the properties stated in \cite[Lemma 49.11.7]{stacks-project}.
\end{proposition}
The corresponding statement for $\E$-rings is as follows.
\begin{proposition}
\label{univ-fsyn-E}
Let $f: A \to B$ be a morphism of $\E$-rings. Then $f$ is finite syntomic if and only if, Zariski locally on $\Spec{\,\pi_0(A)}$, there exists an integer $d \ge 1$ such that the induced morphism $\Spec{B} \to \Spec{A}$ of derived schemes is a base change of the universal finite syntomic morphism $p_d: V_d \to U_d$ in Proposition~\ref{univ-fsyn}.
\end{proposition}
\begin{proof}
By definition, $f$ is finite syntomic if and only if $B$ is flat over $A$ and $\pi_0(A)\to \pi_0(B)$ is finite syntomic. By Proposition~\ref{univ-fsyn}, the latter condition is equivalent to the existence of the stated local charts on $\Spec{\,\pi_0(A)}$. Since flatness is preserved by base change and can be checked on the corresponding homotopy Cartesian squares, these charts lift to the derived morphism $\Spec{B} \to \Spec{A}$. Conversely, any such local pullback of $p_d$ is flat and has finite syntomic $\pi_0$, so $f$ is finite syntomic. \qed
\end{proof}

\begin{corollary} 
\label{cobase-change}
For any $\E$-ring $A$, the $\infty$-category of finite syntomic $\E$-algebras is closed under cobase change. \qed
\end{corollary}

\subsection{Excisive approximation of augmented algebras}
\label{sec:non-unital}
The excisive approximation of the category of augmented $V$-algebras is described as follows:
\begin{proposition}
\label{excisive-nonunital}
Let $V$ be a unital commutative algebra. Then the following hold:
\begin{enumerate}
 \item For any augmented $V$-algebra $R$, the unit
 \[
   P^1(R) \to V \oplus L_{R/V}
 \]
 is a weak equivalence in $P^1(\CAlg_{V//V})$.
 \item Let $A$ be a commutative $V$-algebra. Given a square-zero extension 
\[
  0 \to M \to     \tilde{A} \to A \to 0.
\]
Then, in the excisive approximation $P^1(\CAlg_{A//A})$, 
\[
    \Sigma_{A}^n P_{A}(\tilde{A}) \to A \oplus M [n]          
\] 
is a weak equivalence for any integer $n$. 
\end{enumerate}
\end{proposition}
\begin{proof}
Since $\Sigma^{\infty}_{A}A$ is contractible, the induced morphism $M \to \Sigma^\infty_{A}\tilde{A}$ is a weak equivalence. The unit $ P^1_{A}(\tilde{A}) \to \Omega^\infty_{A} \Sigma_{A}^\infty \tilde{A} \simeq A \oplus L_{\tilde{A} / A}$ is a weak equivalence. Therefore, $M \to L_{\tilde{A} / A}$ is a weak equivalence. \qed
\end{proof}

A commutative non-unital $V$-algebra $I$ is said to be {\it homotopically square-zero} if the product $I \otimes_{V} I \to I$ is null-homotopic.
\begin{proposition}
\label{square-zero} Let $I$ be a commutative non-unital $V$-algebra. Then $I$ is excisive if and only if $I$ is homotopically square-zero.
\end{proposition}
\begin{proof}
This follows directly from the properties of exact functors and Ken Brown's lemma applied to $A \otimes_V A \simeq A \times_V A$. \qed
\end{proof}

\subsection{Finite derived algebraic cobordism of non-unital algebras}
\label{sec:MGL}
The classical motivic cobordism $\MGL$ enforces strict $\mathbb{A}^1$-invariance. Over mixed characteristic bases, this destroys nilpotent thickenings and makes linear approximations trivial. To retain infinitesimal deformation data, we work with unlocalized derived algebraic cobordism $\mathrm{dMGL}$ \cite{LS16, Ann21}.

By utilizing spans of quasi-smooth morphisms, derived algebraic cobordism unconditionally preserves higher intersection data. In this section, we construct a finite syntomic analogue of non-unital derived cobordism by directly using the unlocalized $\infty$-groupoid of finite syntomic $\E$-algebras, bypassing the $\mathbb{A}^1$-localization entirely.

\begin{definition}
Let $V$ be a unital algebra. For any augmented $V$-algebra $A$, let $\mathrm{FSyn}^{\rm aug}(A)$ denote the unlocalized $\infty$-groupoid of finite syntomic $A$-algebras, $\mathrm{FSyn}_*(A)$ the homotopy fiber of $\mathrm{FSyn}(A) \to \mathrm{FSyn}(V)$, and 
\[
  \mathrm{FSyn}^{\rm aug}_*: \CAlg_{V//V} \to \mathcal{S}
\] 
its left Kan extension. The {\it finite non-unital derived algebraic cobordism} $\mathrm{dMGL}^{\rm nu,f}$ over $\Spec{V}$ is defined to be the stabilization 
\[
   \Sigma^\infty \mathrm{FSyn}_*^{\rm aug}: \CAlg_{V//V} \to \Sp.
\] 
\end{definition}

Crucially, because we do not impose $\mathbb{A}^1$-localization, the functor $\mathrm{FSyn}_*^{\rm aug}$ faithfully detects the infinitesimal thickenings governed by the derived cotangent complex. Applying Corollary~\ref{Universal-cpt}, $\mathrm{dMGL}^{\rm nu,f}$ factors through the excisive approximation $P^1(\CAlg_{V//V})$.

\begin{lemma} 
\label{key-lemma} Given a square-zero extension $\phi: \tilde{A} \to A$, let $M$ denote the kernel of $\phi$. Write $\Sigma^1_{A} \tilde{A}= \Sigma^1 P^1_A(\tilde{A})=A \oplus M[1]$. Then we can explicitly construct a model
\[
 \tilde{B}= B \oplus (B \otimes_{A} M[1]) \simeq B \otimes_{A} (A \oplus M[1]). 
\] 
This lifting $\tilde{B}$ satisfies the following properties:
\begin{enumerate}
 \item The induced morphism $\tilde{B} \otimes_{\Sigma^1_{A} \tilde{A} } A \to B$ is a weak equivalence.
 \item The $\Sigma^1_{A} \tilde{A}$-algebra $\tilde{B}$ is finite syntomic over $\Sigma^1_{A} \tilde{A}$. 
\end{enumerate}
\end{lemma}
\begin{proof}
Property (1) follows from $B \otimes_A (A \oplus M[1]) \otimes_{A \oplus M[1]} A \simeq B$. Property (2) follows from Corollary~\ref{cobase-change}, as the base change $A \oplus M[1] \to \tilde{B}$ along the finite syntomic algebra homomorphism $A \to B$ is also finite syntomic. \qed  
\end{proof}

We recall the following $\infty$-categorical deformation theory:
\begin{theorem}[{\rm cf.}~\cite{HA}, pp.1349--1350, Remark 7.4.2.2 and Remark 7.4.2.3]
\label{Kodaira-Spencer}
Let $A$ be a unital $\E$-ring and $ \tilde{A} \overset{\phi}{\to} A$ a square-zero extension with homotopy kernel $M$. Let $B$ be an $A$-algebra. If $\Ext_{B}^{2}(L_{B/A},\, B \otimes_A M)=0$, there exists an $\tilde{A}$-algebra $\tilde{B}$ such that the induced homomorphism $\tilde{B} \otimes_{\tilde{A}} A \to B$ is an isomorphism. Further, if $\Ext_{B}^{1}(L_{B/A},\, B \otimes_A M)=0$, the flat $\tilde{A}$-algebra $\tilde{B}$ is uniquely determined up to quasi-isomorphism. \qed
\end{theorem}

\begin{proposition}
\label{1to1} 
Let $\phi: \tilde{A} \to A$ be a square-zero extension of commutative discrete $V$-algebras, and let $M$ denote the kernel of $\phi$. Then the functor $(-) \otimes_{ \Sigma^1 P^1_{A}(\tilde{A})} A$ induces a weak equivalence
\[
 (-) \otimes_{ \Sigma^1 P^1_{A}(\tilde{A})} A : \mathrm{FSyn}^{\rm aug}( \Sigma^1 P^1_{A}(\tilde{A}) ) \to \mathrm{FSyn}^{\rm aug}(A)  
\]
of unlocalized spaces. 
\end{proposition}
\begin{proof}
Put
\[
R=\Sigma^1 P^1_A(\tilde{A})\simeq A\oplus M[1].
\]
By Lemma~\ref{key-lemma}, the functor
\[
 s:\mathrm{FSyn}^{\rm aug}(A)\to \mathrm{FSyn}^{\rm aug}(R),\qquad
 B\mapsto B\otimes_A R
\]
is well-defined, and the composite $(-)\otimes_R A\circ s:\mathrm{FSyn}^{\rm aug}(A) \to \mathrm{FSyn}^{\rm aug}(A) $
is equivalent to the identity functor on $\mathrm{FSyn}^{\rm aug}(A)$.

Conversely, let $C$ be a finite syntomic $R$-algebra and put
$B=C\otimes_R A$. By Corollary~\ref{cobase-change}, $B$ is finite
syntomic over $A$, hence flat over $A$. Since $C$ is flat over $R$ and
$\pi_0(R)=A$, the canonical morphism $
 B\otimes_A R \to C $ induces an isomorphism on $\pi_0$. For each $i\ge 0$, one has
\[
 \pi_i(B\otimes_A R)\cong \pi_0(B)\otimes_A \pi_i(R)
 \cong \pi_0(C)\otimes_A \pi_i(R)
 \cong \pi_i(C),
\]
where the first isomorphism uses that $B$ is flat over $A$, and the
last one uses that $C$ is flat over $R$. Therefore $B\otimes_A R \to C$
is a weak equivalence of connective $R$-algebras. Thus, the composition $s\circ((-)\otimes_R A)$
is also equivalent to the identity. Hence $(-)\otimes_R A$ is an
equivalence of unlocalized spaces. \qed
\end{proof} 

We also recall the Milnor exact sequence for towers of spectra: for any
inverse system $\{X_n\}_{n\ge 1}$, one has a short exact sequence
\[
0\to \varprojlim\nolimits^1_n \pi_{*-1}(X_n)\to
\pi_*(\varprojlim_n X_n)\to \varprojlim_n \pi_*(X_n)\to 0.
\]

\begin{theorem}
\label{mainthm} Let $V$ be a perfectoid valuation ring and $A$ be an integral perfectoid $V$-algebra. The tilting functor $(-)^\flat : \CAlg_A \to \CAlg_{A^\flat}$ induces a weak equivalence
\[
\mathrm{dMGL}^{\rm nu,f}(A) \to \mathrm{dMGL}^{\rm nu,f}(A^\flat)
\]
of spectra.
\end{theorem}
\begin{proof}
Put $F:=\mathrm{FSyn}^{\rm aug}_*:\CAlg_{V//V}\to\mathcal{S}$. Because $F$ is unlocalized, it accurately captures infinitesimal deformations and factors through $P^1(\CAlg_{V//V})$; hence $F$ is reduced and excisive.

Choose a topologically nilpotent element $\omega\in V$, and let $\omega^\flat\in V^\flat$ denote its tilt.

For each $n\ge 1$, write $A_n=A/\omega^nA$ and $A_n^\flat=A^\flat/(\omega^\flat)^nA^\flat$. For each $n\ge 2$, the short exact sequence $0\to \omega^{n-1}A/\omega^nA\to A_n\to A_{n-1}\to 0$ is a square-zero extension. Applying Proposition~\ref{1to1} to $A_n\to A_{n-1}$, we obtain a weak equivalence $F(\Sigma^1_{A_{n-1}}A_n)\to F(A_{n-1})$.

Since $F$ is reduced and excisive, $F(\Sigma X)\simeq \Omega F(X)$. Therefore, $\Omega F(A_n)\simeq F(A_{n-1})$, yielding $F(A_n)\simeq \Sigma^{n-1}F(A_1)$. Using $A_1\cong A_1^\flat$, we get weak equivalences $F(A_n)\to F(A_n^\flat)$ for all $n\ge 1$.

By Proposition~\ref{univ-fsyn-E}, finite syntomic algebras are Zariski locally obtained from the finitely presented universal families of \cite[Section 49.11]{stacks-project}. For the derived descent and completion formalism that allows these finite-presentation charts to be used for connective $\E$-rings and their nilpotent thickenings, see Lurie's derived algebraic geometry \cite{DAGXI,DAGXII}. Hence \cite[Proposition 32.6.1]{stacks-project} implies that compatible systems of finite syntomic algebras over the towers $\{A_n\}_n$ and $\{A_n^\flat\}_n$ algebraize uniquely, so the canonical maps
\[
F(A)\to \varprojlim_n F(A_n),\qquad
F(A^\flat)\to \varprojlim_n F(A_n^\flat)
\]
are weak equivalences.

Applying the preceding Milnor exact sequence to the towers
$X_n=F(A_n)$ and $X_n=F(A_n^\flat)$, we obtain short exact sequences
\[
0\to \varprojlim\nolimits^1_n\pi_{*-1}(F(A_n)) \to
\pi_*(\varprojlim_n F(A_n))\to \varprojlim_n\pi_*(F(A_n))\to 0
\]
and
\[
0\to \varprojlim\nolimits^1_n\pi_{*-1}(F(A_n^\flat)) \to
\pi_*(\varprojlim_n F(A_n^\flat))\to \varprojlim_n\pi_*(F(A_n^\flat))\to 0.
\]
Since $F(A_n)\simeq F(A_n^\flat)$ for all $n\ge 1$, the left and right
terms are isomorphic, hence the middle terms are also isomorphic. Thus
\[
\varprojlim_n F(A_n)\to \varprojlim_n F(A_n^\flat)
\]
is a weak equivalence. Combining this with the canonical identifications above, we obtain that
\[
(-)^\flat:F(A)\to F(A^\flat),
\]
is a weak equivalence, implying that  
\[
\mathrm{dMGL}^{\rm nu,f}(A)\simeq \Sigma^{\infty}F(A) \simeq
\Sigma^{\infty}F(A^\flat) \simeq \mathrm{dMGL}^{\rm nu,f}(A^\flat).
\]
\qed
\end{proof}

\subsection{Finite derived algebraic cobordism of almost algebras}
\label{sec:alMGL} 
Following the Smith-ideal formulation of almost mathematics in~\cite{kato2023mathematics}, almost mathematics is described by base rings and their idempotent ideals. We consider the case where the base ring $V$ has an idempotent ideal $\mathfrak{m}$. We strictly assume that $\mathfrak{m} \subsetneq V$ is a proper idempotent ideal, and $\tm= \mathfrak{m} \otimes_V \mathfrak{m}$ is a flat $V$-module.

\begin{definition}
Let $A$ be an almost $V$-algebra. For an almost $A$-algebra $B$, write
\[
  B_{!!}:=V\amalg_{\tm}(\tm\otimes_V B)
\]
for its almost unitalization. We say that $B$ is almost finite syntomic
over $A$ if the induced morphism $A_{!!}\to B_{!!}$ is finite syntomic.
Let $\CAlg^{\rm alFSyn}_{A}$ be the full subcategory of almost
$A$-algebras spanned by almost finite syntomic $A$-algebras, and let
$\mathrm{alFSyn}(A)$ denote its underlying $\infty$-groupoid. The {\it
finite almost derived algebraic cobordism} $\mathrm{dMGL}^{a,\rm fin}(A)$
is defined to be the stabilization $\Sigma^{\infty}\mathrm{alFSyn}(A)$.
\end{definition}

To relate this almost construction to the non-unital one, use the
almost unitalization functor of~\cite{kato2023mathematics}. For every
almost $V$-algebra $B$ there is a natural equivalence
\[
  \tm\otimes_V B_{!!}\simeq \tm\otimes_V B.
\]
This is the point needed below: after passing from an almost algebra to
its almost unitalization, tensoring again with $\tm$ recovers the same
almost algebra. Hence the almost finite syntomic groupoid is a homotopy
retract of the finite syntomic groupoid associated with the non-unital
algebra $V\oplus(\tm\otimes_V A)$: the inclusion is given by almost
unitalization, and the retraction is induced by tensoring with $\tm$.
After stabilization,
$\mathrm{dMGL}^{a,\rm fin}(A)$ is therefore a homotopy direct factor of
$\mathrm{dMGL}^{\rm nu,f}(V\oplus(\tm\otimes_V A))$. This gives the
following corollary of Theorem~\ref{mainthm}.
\begin{corollary}
\label{almost-main} Let $V$ be an integral perfectoid valuation ring and $\mathfrak{m}$ the proper idempotent ideal of topological nilpotent elements. For any integral perfectoid $V$-algebra $A$, the tilting functor $(-)^\flat: \CAlg_{V} \to \CAlg_{V^\flat}$ induces a weak equivalence
\[
  (-)^\flat: \mathrm{dMGL}^{a,\rm fin}(A) \to \mathrm{dMGL}^{a,\rm fin}(A^\flat)
\]
of finite almost derived algebraic cobordism spectra.
\end{corollary}
\begin{proof}
By the preceding paragraph, there is a homotopy retract exhibiting
\[
  \mathrm{dMGL}^{a,\rm fin}(A)
\]
as a homotopy direct factor of
\[
  \mathrm{dMGL}^{\rm nu,f}(V\oplus(\tm\otimes_V A)).
\]
The tilting equivalence of Theorem~\ref{mainthm} is compatible with this
retract, so passing to the direct factor gives the asserted weak equivalence.
\qed
\end{proof}

\section{Approximation of algebraic cobordism to \texorpdfstring{$K$-theory}{K-theory}}
\label{sec:functor} In this section, we apply Goodwillie calculus to
algebraic cobordism and $K$-theory. In Section~\ref{sec:app-fun}, we
define approximation functors. In Section~\ref{sec:MGL2K}, we apply
them to algebraic cobordism and $K$-theory, and prove Bott periodicity
and Gabber rigidity.

\subsection{Approximation functors}
\label{sec:app-fun}
Let $\alpha:F \to G$ be a natural transformation of functors
$F,G:\mathcal{C} \to \mathcal{D}$, and let
$\mathrm{cok}\,\alpha$ be its homotopy cokernel. Under suitable
conditions, let $P^n_{G}(F):\mathcal{C} \to \mathcal{D}$ be the fiber of
\[
G \to \mathrm{cok}\,\alpha \to P^n(\mathrm{cok}\,\alpha).
\]
Then $P^0_G(F) \to G$ is an equivalence. Thus we get a tower
\[
   \cdots  \to    P^n_G(F) \to P^{n-1}_G(F) \to \cdots \to P^1_G(F) \to P^0_G(F) \simeq G : \mathcal{C} \to \mathcal{D}.
\]
We call $P^n_{G}(F)$ the {\it $n$-Goodwillie approximation} of $F$ to $G$.

\begin{proposition}
Let $F \to G: \mathcal{C} \to \mathcal{D} $ be a natural transformation
of functors from $\mathcal{C}$ to $\mathcal{D}$. If $G$ is $n$-excisive,
the $n$-Goodwillie approximation $P^n_G(F) $ is $n$-excisive.
\end{proposition}
\begin{proof}
The full subcategory $\Exc^n(\mathcal{C},\,\mathcal{D})$ is closed under small limits in $\Fun(\mathcal{C},\,\mathcal{D})$. In particular, the fiber of $G
\to P^n(\mathrm{cok}(\alpha))$ is $n$-excisive.  \qed
\end{proof}

\subsection{Approximation to \texorpdfstring{$\K$-theory}{K-theory}}
\label{sec:MGL2K} Let $\CAlg(\MSp_\infty)$ be the $\infty$-category of
motivic $\E$-rings. If we forget the unit, a motivic $\E$-ring becomes a
motivic non-unital $\E$-ring. Since algebraic cobordism is universal
among oriented motivic spectra, there is a canonical morphism
$\alpha:\mathrm{MGL} \to \K$.

For the universal morphism $\alpha: \mathrm{MGL} \to \K$, the
$n$-Goodwillie approximation $P^n_\K(\mathrm{MGL})$ inherits properties
from $\K$.
\begin{proposition}
Let $S$ be a scheme and $\alpha: \mathrm{MGL} \to \K $ the canonical
morphism of oriented motivic $\E$-rings. Then, for each $n \ge 0$, the
$n$-Goodwillie approximation $P^n_\K( \mathrm{MGL})$ is an augmented
$\K$-algebra and periodic motivic $\E$-ring.
\end{proposition}
\begin{proof}
Since $\K$ has Bott periodicity, $P^0_\K(\MGL)$ is
periodic and is the trivial $\K$-augmented $\K$-algebra. We prove that
$P^n_\K(\MGL)$ is periodic and a $\K$-augmented algebra for each $n$ by
induction on $n \ge 0$.  The homotopy Cartesian square
\[
 \xymatrix@1{ P^{n}(\mathrm{cok} \alpha)  \ar[r] \ar[d]& P^{n-1}(\mathrm{cok} \alpha )  \ar[d]\\
  {*} \ar[r] & R^n (\mathrm{cok} \alpha ) }
\]
induces a homotopy Cartesian square
\[
 \xymatrix@1{  P_\K^n(\MGL) \ar[r] \ar[d] & P_\K^{n-1}(\MGL)  \ar[d]\\
    \K  \ar[r]& R^n_\K(\MGL) 
} 
\]
of motivic $\E$-rings. Since $P^{n-1} \circ R^n$ is contractible for $n
\ge 1$, $ R^n_\K(\MGL) \to P_\K^{n-1}( \MGL) \simeq \K$ is homotopic to
the augmentation to $\K$. Therefore, $R_\K^n(\MGL)$ has a
$\mathrm{PMGL}$-algebra structure implying that it is
periodic. Therefore, by the assumption of the induction, $P^n_\K(\MGL)$
is periodic and $\K$-augmented. \qed
\end{proof}

\begin{corollary}
\label{split} Let $S$ be a scheme. The canonical morphism $p: \MGL \to
\mathrm{PMGL} $ induces a homotopically split monomorphism
\[
   P^n(p): P^n_\K(\MGL) \to  P^n_\K(\mathrm{PMGL})  
\]   
for each $n \ge 0$. 
\end{corollary}
\begin{proof}
Since $P^n_\K(\MGL)$ is periodic, $\MGL \to P^n_\K(\MGL) $ factors
through $\mathrm{PMGL}$. Therefore $ P^n(p): P^n_\K(\MGL) \to
P^n_\K(\mathrm{PMGL}) $ is homotopically split.  \qed
\end{proof}

\begin{lemma}
\label{ideal} Let $A$ be an $\E$-ring object of a symmetric monoidal
$\infty$-category and let $I$ denote the homotopy fiber of an augmentation
$\tilde{A} \to A$ of $A$-algebras. Then the relative cotangent complex
$L_{A/P^1_A(\tilde{A})}[-1] $ is a homotopy fiber of the induced augmentation 
\[
A \to A \otimes_{P^1_A(\tilde{A})}A.
\] 
\end{lemma}
\begin{proof} 
The cofiber sequence $I \to \tilde{A}\to A$ induces $P^1_A(\tilde{A})
\simeq \Omega_{A}\Sigma_{A}(P^1_A(\tilde{A})) \simeq A \oplus
L_{A/P^1_A(\tilde{A})}[-1] \simeq A \oplus P^1_{A}(I)$. 
The homotopy biCartesian square
\[
 \xymatrix@1{
  P^1_A(\tilde{A}) \ar[r] \ar[d] &  A \ar[d] \\
  A  \ar[r]&  A\otimes_{P^1_A(\tilde{A})} A 
 }
\]
induces a weak equivalence 
\[
   L_{A/P^1_A(\tilde{A})}[-1] \to \mathrm{fib}(  A \to A \otimes_{P^1_A(\tilde{A})}A).
\]
\qed
\end{proof}

\begin{theorem}
\label{theoremB}
Let $f: X\to Y$ be a morphism of motivic spectra. Assume that $f$ induces a weak equivalence $f^*: \K(Y) \to \K(X)$ of homotopy $K$-theory spectra. Then this weak equivalence lifts to a weak equivalence
\[
P^n(f^*): P^n_\K(\mathrm{PMGL})(Y) \to   P^n_\K(\mathrm{PMGL})(X) 
\] 
of $n$-Goodwillie approximations of the periodic algebraic cobordisms
for each $n \ge 0$.
\end{theorem}
\begin{proof}
For each $n\ge 1$, by construction of relative Goodwillie approximations,
there is a homotopy Cartesian square
\[
 \xymatrix@1{
   P^n_\K(\mathrm{PMGL}) \ar[r] \ar[d] & P^{n-1}_\K(\mathrm{PMGL}) \ar[d] \\
   \K \ar[r] & R^n_\K(\mathrm{PMGL}).
 }
\]
Hence it is enough to prove that, for each $n\ge 1$,
\[
R^n_\K(\mathrm{PMGL})(Y)\to R^n_\K(\mathrm{PMGL})(X)
\]
is a weak equivalence.

For $n=1$, Lemma~\ref{ideal} gives that
$L_{\K/P^1_\K(\mathrm{PMGL})}[-1]$ is a homotopy retract of $\K$. Therefore
\[
L_{\K(Y)/P^1_\K(\mathrm{PMGL})(Y)}[-1]\to
L_{\K(X)/P^1_\K(\mathrm{PMGL})(X)}[-1]
\]
is a weak equivalence, and so are
\[
P^1_\K(\mathrm{PMGL})(Y)\to P^1_\K(\mathrm{PMGL})(X),\qquad
R^1_\K(\mathrm{PMGL})(Y)\to R^1_\K(\mathrm{PMGL})(X).
\]

Now let $n\ge 2$. Let $\mathcal{H}^n_\K$ denote the full subcategory of
$n$-homogeneous functors in the ambient functor category. By Goodwillie
theory (\cite[Chapter 6]{HA}), $\mathcal{H}^n_\K$ is a stable
$\infty$-category. Hence the unit map
\[
R^n_\K(\mathrm{PMGL})\to \Omega^1_\K\Sigma^1_\K R^n_\K(\mathrm{PMGL})
\]
is a weak equivalence, and the augmentation
$\varepsilon_n:R^n_\K(\mathrm{PMGL})\to \K$ is a homotopically square-zero
extension.

Applying Lemma~\ref{ideal} to $\varepsilon_n$ objectwise at $Y$ and $X$,
the relative cotangent complexes
\[
L_{\K(Y)/R^n_\K(\mathrm{PMGL})(Y)}\to
L_{\K(X)/R^n_\K(\mathrm{PMGL})(X)}
\]
are identified with the corresponding fibers of the augmentation squares.
Since $f^*: \K(Y)\to\K(X)$ is a weak equivalence, base change along $f^*$
induces an equivalence on module categories; therefore the above map of
relative cotangent complexes is a weak equivalence. By the square-zero
description of $\varepsilon_n$, this implies
\[
R^n_\K(\mathrm{PMGL})(Y)\to R^n_\K(\mathrm{PMGL})(X)
\]
is a weak equivalence.

Finally, applying the first homotopy Cartesian square objectwise at $Y$ and
$X$, and using induction on $n$, we conclude that
\[
P^n_\K(\mathrm{PMGL})(Y)\to P^n_\K(\mathrm{PMGL})(X)
\]
is a weak equivalence for any $n\ge 0$. \qed
\end{proof}

\begin{corollary}
Given a morphism $f:X\to Y$ of noetherian $S$-schemes of finite type that induces a weak equivalence $f^*: \K(Y) \to \K(X)$ of homotopy $K$-theory
spectra, one has a weak equivalence
\[
P^n(f^*): P^n_\K(\mathrm{MGL})(Y) \to   P^n_\K(\mathrm{MGL})(X) 
\] 
of $n$-Goodwillie approximations of algebraic cobordism for each $n
\ge 0$.   
\end{corollary}
\begin{proof}
This is a direct result of Theorem~\ref{theoremB} and
Corollary~\ref{split}. \qed
\end{proof}

We give an analogue of the Gabber rigidity theorem for nilpotent
approximation in~\cite{Kato2023nilpotent}:

\begin{theorem}[{\rm cf.}~\cite{Kato2023nilpotent} Theorem 5.12]
\label{Rigidity} Let $A$ be a noetherian commutative unital ring and $I$ an ideal
such that $(A,\,I)$ is a Henselian pair. Assume that the characteristic
of the residue ring $A/I$ is a positive prime $p$, and that $\ell$ is a
positive integer invertible in $A/I$. Then the closed immersion
$i^*:\Spec{A/I} \to \Spec{A}$ induces a weak equivalence of
$n$-Goodwillie approximations
\[
 i^*:  P^n_{\K/\ell}(\mathbf{MGL}/{\ell})( \Spec{A} ) \to 
P^n_{\K/\ell}(\mathbf{MGL}/{\ell})( \Spec{A/I}), 
\]
where $\mathbf{MGL}/{\ell}$ and $\K/\ell$ denote the mod-$\ell$ algebraic
cobordism and $K$-theory, respectively. \qed
\end{theorem}

\begin{remark}
By Proposition~\ref{square-zero}, excisive approximation of the homotopy
cofiber $\mathrm{PMGL} \to \K $ and one-nilpotent (square-zero)
approximation in \cite{Kato2023nilpotent} are equivalent formulations.
\end{remark}

\bibliographystyle{alphadin}
{\raggedright\sloppy\bibliography{bibkato}\par}
\nocite{Jardine, Jardine2}
\nocite{K-MGL}
\nocite{kato2023mathematics}
\nocite{MV}
\nocite{zbMATH03811819}
\nocite{zbMATH07335469}
\nocite{zbMATH05706074}
\nocite{zbMATH02095715}
\nocite{zbMATH01787461}
\nocite{Hirschhorn}
\nocite{Hoveybook}
\nocite{HT}
\nocite{HA}
\nocite{VoeH}
\nocite{V}
\nocite{preygel}
\end{document}